\documentclass[a4paper,11pt]{article}
\usepackage{graphicx}
\usepackage{amssymb}
\usepackage[french, english]{babel}
\usepackage[utf8]{inputenc}
\usepackage{amsmath,amsfonts,amssymb,amsthm}
\usepackage[T1]{fontenc}
\usepackage{comment}
\usepackage{fullpage}
\usepackage{hyperref}
\usepackage{mathrsfs}
\usepackage{relsize}
\usepackage{tikz}
\usepackage{bbm}
\usepackage{appendix}
\usepackage{pifont}
\usetikzlibrary{patterns}

\usepackage{needspace}

\definecolor{darkgreen}{rgb}{0.0, 0.5, 0.0}

\newcommand{\eps}{\varepsilon}

\newcommand{\down}{\text{down}}
\newcommand{\out}{\text{out}}
\newcommand{\dif}{\text{diff}}
\newcommand{\red}{\text{red}}

\newcommand{\vol}{\text{vol}}
\newcommand{\upp}{\text{up}}

\usepackage{physics}

\theoremstyle{definition}

\newtheorem{thm}{Theorem}

\newtheorem{defn}{Definition}%[section]

\newtheorem{prop}[defn]{Proposition}

\newtheorem{lem}[defn]{Lemma}

\tikzstyle{every node}=[circle, draw, fill=black!50, inner sep=0pt, minimum width=4pt]
\tikzstyle{white}=[circle, draw, fill=black!0, inner sep=0pt, minimum width=4pt]
\tikzstyle{bigwhite}=[circle, draw, fill=black!0, inner sep=0pt, minimum width=10pt]
\tikzstyle{dual}=[circle, draw=blue, fill=black!0, inner sep=0pt, minimum width=4pt]
\tikzstyle{fat}=[circle, draw, fill=red!50, inner sep=0pt, minimum width=8pt]
\tikzstyle{fat_bis}=[circle, draw, fill=blue!50, inner sep=0pt, minimum width=8pt]
\tikzstyle{fat_ter}=[circle, draw, fill=green!50, inner sep=0pt, minimum width=8pt]
\tikzstyle{rouge}=[circle, draw, fill=red, inner sep=0pt, minimum width=7pt]
\tikzstyle{bleu}=[circle, draw, fill=blue, inner sep=0pt, minimum width=7pt]
\tikzstyle{petitrouge}=[circle, draw, fill=red, inner sep=0pt, minimum width=4pt]
\tikzstyle{petitbleu}=[circle, draw, fill=blue, inner sep=0pt, minimum width=4pt]
\tikzstyle{texte}=[draw=none, fill=none]

\title{\bf{Finding large expanders in graphs: from topological minors to induced subgraphs}}
\author{Baptiste Louf\thanks{Uppsala Universitet.}  \and Fiona Skerman\footnotemark[1]}

\begin{document}

\maketitle

\paragraph{Abstract.}  
In this paper, we consider {a structural and  geometric property} of graphs, namely the presence of large expanders. {The problem of finding such structures} was first considered by~Krivelevich~[SIAM J.\ Disc.\ Math.\ 32 1 (2018)]. Here, we show that the problem of finding a large induced subgraph that is an expander can be reduced to the simpler problem of finding a topological minor that is an expander. Our proof is constructive, which is helpful in an algorithmic setting.

We also show that every large subgraph of an expander graph contains a large subgraph which is itself an expander. 

\section{Introduction}
\setlength{\parindent}{0pt}

Expander graphs are important objects in graph theory. They have a wide variety of properties: for instance, the random walk mixes very fast~\cite{AKS87,Gil98}, and the eigenvalues of their Laplacian are well-separated~\cite{Al86, Che70}. In computer science, they are used for clustering with the expander decomposition technique (see for instance~\cite{KVV04}). We refer to~\cite{HLW06} for an extensive survey of expander graphs and their applications.\\ 

Unfortunately, the conditions for a graph to be an expander are quite restrictive, as the edge expansion of any subset of vertices should be high enough. In particular, the Cheeger constant of a graph is influenced by every "microscopic" part of the graph: for example adding a disjoint edge causes the Cheeger constant to drop to zero. There is a much larger set of graphs that shares macroscopic features with expanders and is more robust to small perturbations in the edge set. A natural such set is all graphs which contain linear sized expanders as induced subgraphs. Here, by large, we mean containing a constant fraction of the edges of the graph. The problem of finding linear sized expander subgraphs has already been considered in~\cite{Kri18}, where a sufficient local sparsity condition was given. This follows previous works on "weak" expander subgraphs in which the expansion factor is allowed to depend on the size of the graph~\cite{Mon15,SS15}.\\

However, finding an induced expander subgraph can be a difficult endeavour, as the only operation that is permitted is to delete a vertex and all its adjacent edges. Our main result is that this problem is actually equivalent to finding a topological minor that is an expander, which allows for more freedom in the operations to be performed. More precisely, the following theorem holds.

\needspace{12\baselineskip}
\begin{thm}\label{thm_topminor_to_induced}
For all $\kappa, \alpha>0$, and for all $0<\alpha'<\alpha$, there exists a $\kappa'>0$ such that the following holds for every (multi)graph $G$.

If there exists a graph $H$ satisfying the following conditions:
\begin{itemize}
\item $e(H)\geq \alpha e(G)$,
\item $H$ is a topological minor of $G$,
\item $H$ is a $\kappa$-expander,
\end{itemize}
then there exists a graph $H^*$ satisfying the following conditions:
\begin{itemize}
\item $e(H^*)\geq \alpha' e(G)$,
\item $H^*$ is an induced subgraph of $G$,
\item $H^*$ is a $\kappa'$-expander.
\end{itemize}
\end{thm}

The proof is constructive, which implies that if we are given an algorithm to construct a topological minor of a graph that is an expander, we have an algorithm to construct an induced subgraph of that graph that is an expander. This result can be helpful to study some models of graphs where the topology is an important feature, such as \emph{combinatorial maps}. To illustrate this, Theorem~\ref{thm_topminor_to_induced} will be applied in an upcoming article by the first author~\cite{Lou21} concerning large expanders in uniform unicellular maps of high genus.\\

The proof of the main  theorem strongly relies on the following theorem of independent interest, that states that every large subgraph of an expander contains a large subgraph that is itself an expander.

\begin{thm}\label{thm_big_subgraph_contains_expander}
Take $0<\kappa\leq 1$ and $0<\eps<\frac{\kappa}{6}$, and $G$  a $\kappa$-expander. Then for any subgraph $H$ of $G$ satisfying $e(H)\geq (1-\eps) e(G)$, there exists an induced subgraph $H^*$ of $H$ satisfying
\begin{itemize}
\item $e(H^*)\geq \left(1-\frac{6\eps}{\kappa} \right)e(G)$,
\item $H^*$ is a $\frac{\kappa}{3}$-expander.
\end{itemize}
\end{thm}

\needspace{6\baselineskip}
\section{Definitions}

A \emph{graph} $G$ is the data of a set of \emph{vertices} $V(G)$ and a multiset $E(G)$ of \emph{edges} whose elements are unordered pairs of (non-necessarily distinct) vertices in $V(G)$. Note that our definition of a graph allows loops and multiple edges, it is also common to call these objects multigraphs. For any graph $G$, we write $e(G)$ for its number of edges. We also write $e_G(X)$ and $e_G(X,Y)$ for the number of edges in $X$ and the number of edges between vertex sets $X$ and $Y$ respectively.  If $X\subset V(G)$, then we write $\overline{X}:= V(G)\setminus X$. {The \emph{degree} of a vertex is the number of non-loop edges it belongs to, plus twice the number of loops it belongs to.} If $X\subset V(G)$, then we write $\vol(X)=\sum_{v\in X} \deg(v)$. Given a vertex $v$ of degree $2$ not incident to a loop, we call \emph{smoothing}~$v$ the operation that consists in replacing $v$ and its two incident edges by a single edge (see Figure~\ref{fig_smoothing}).\\

\begin{figure}
    \centering
    \includegraphics[scale=0.7]{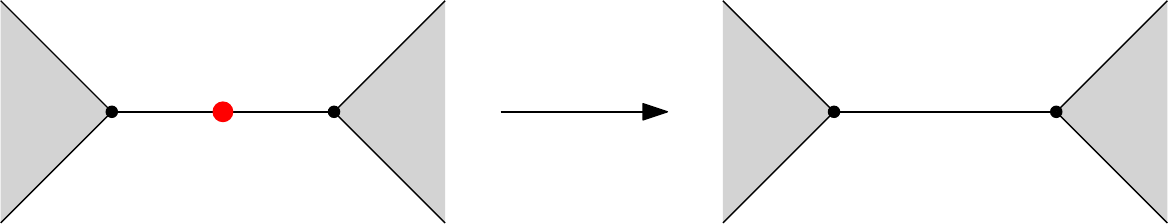}
    \caption{Smoothing a vertex of degree $2$.}
    \label{fig_smoothing}
\end{figure}

If $G$ is a graph, an \emph{induced subgraph} of $G$ is a graph obtained from $G$ by deleting some of its vertices. A \emph{subgraph} of $G$ is a graph obtained from $G$ by deleting some of its vertices and edges. A \emph{topological minor} of $G$ is a graph obtained from $G$ by deleting some of its vertices and edges, and smoothing some of its vertices. {For $H$ a subgraph of $G$ we write $G \setminus H$ to denote the graph with vertex set $V(G)$ and edge set $E(G)\setminus E(H)$.}\\

The following fact is folklore:
\begin{lem}\label{lem_ordering_operations}
Any subgraph of a graph can be obtained by first applying vertex deletions, then edge deletions.
Any topological minor of a graph can be obtained by first applying vertex deletions, then edge deletions, then smoothings.
\end{lem}

The \emph{edge-expansion} of the vertex set $X \subset V(G)$ is defined as
\[h_G(X):=\frac{e_G(X,\overline{X})}{\min(\vol_G(X),\vol_G(\overline{X}))}\]
where $e_G(X,\overline{X})$ is the number of edges of $G$ that have one endpoint in $X$ and one endpoint in~$\overline{X}$. The graph $G$ is said to be a \emph{$\kappa$-expander} if, for all $X\subset V(G)$, we have
\[h_G(X)\geq \kappa.\]

Notice that we actually only have to check the condition above for subsets $X$ such that $G[X]$ is connected where $G[X]$ is the induced subgraph of $G$ obtained by deleting all vertices in $\overline{X}$. The \emph{Cheeger constant} of a graph $G$ is the largest $\kappa$ such that $G$ is a $\kappa$-expander.\\

A \emph{vertex-coloured graph} is a graph that has a certain number of its vertices of degree $2$ coloured in blue. A blue path of size $\ell$ is a path of $\ell$ consecutive blue vertices. If $G$ is a vertex-coloured graph, then we call $\red(G)$ (the "reduced" version of $G$)  the graph obtained from $G$ by smoothing all its blue vertices.\\

An \emph{edge-coloured graph} is a graph that has a certain number of edges coloured in blue. If $G$ is an edge-coloured graph, then we call $\red(G)$  the graph obtained from $G$ by deleting all its blue edges. If $G$ is an edge-coloured graph, and $v\in V(G)$, then we write $d_\red(v)$ for the degree of $v$ in $\red(G)$.

\needspace{6\baselineskip}
\section{Strategy of proof of Theorem~\ref{thm_topminor_to_induced}}

To obtain an expander induced subgraph from an expander topological minor, we proceed in two steps: first we find an expander subgraph of roughly the same size as the topological minor, then we find an expander induced subgraph of roughly the same size as this subgraph. Theorem~\ref{thm_topminor_to_induced} is an immediate corollary of the following two propositions.

\begin{prop}\label{prop_topminor_to_subgraph}
For all $\kappa, \alpha>0$, and for all $0<\alpha'<\alpha$, there exists a $\kappa'>0$ such that the following holds for every graph $G$.

If there exists a graph $H$ satisfying the following conditions:
\begin{itemize}
\item $e(H)\geq \alpha e(G)$,
\item $H$ is a topological minor of $G$,
\item $H$ is a $\kappa$-expander,
\end{itemize}
then there exists a graph $H^*$ satisfying the following conditions:
\begin{itemize}
\item $e(H^*) \geq \alpha' e(G)$,
\item $H^*$ is a subgraph of $G$,
\item $H^*$ is a $\kappa'$-expander.
\end{itemize}
\end{prop}

\begin{prop}\label{prop_subgraph_to_induced}
For all $\kappa, \alpha>0$, and for all $0<\alpha'<\alpha$, there exists a $\kappa'>0$ such that the following holds for every graph $G$.

If there exists a graph $H$ satisfying the following conditions:
\begin{itemize}
\item $e(H)\geq \alpha e(G)$,
\item $H$ is a subgraph of $G$,
\item $H$ is a $\kappa$-expander,
\end{itemize}
then there exists a graph $H^*$ satisfying the following conditions:
\begin{itemize}
\item $e(H^*)\geq \alpha' e(G)$,
\item $H^*$ is an induced subgraph of $G$,
\item $H^*$ is a $\kappa'$-expander.
\end{itemize}
\end{prop}

\needspace{6\baselineskip}
\section{Proof of Theorem~\ref{thm_big_subgraph_contains_expander}}

We start with an easy technical Lemma that will be useful in the future.
\begin{lem}\label{lem_induced_vol}
Let $H$ a graph and $H'$ an induced subgraph of $H$. Let $X\subset V(H')$, we can also consider $X$ as a subset of $V(H)$.
If 
\[\vol_{H'}(X)\leq \vol_{H'}(\overline X)\]
then
\[\vol_{H}(X)\leq \vol_{H}(\overline X).\]
\end{lem}

\begin{proof}
Note that in any graph $G$, 
\[\vol_{G}(X)\leq \vol_{G}(\overline X)\] is equivalent to
\[e_G(X)\leq e_G(\overline X).\] {Hence we may suppose $e_{H'}(X)\leq e_{H'}(\overline{X})$ and it remains only to show $e_{H}(X)\leq e_{H}(\overline{X})$.} Since $X\subset V(H')$ and $H'$ is an induced subgraph of $H$, we have
\[e_{H'}(X)=e_H(X).\]
On the other hand 
\[e_{H'}(\overline X)=e_{H'}(V(H')\setminus X)=e_{H}(V(H')\setminus X)\]
where the last equality holds because $H'$ is an induced subgraph of $H$. We also have
\[e_{H}(\overline X)=e_{H}(V(H)\setminus X).\]
Since $V(H')\setminus X\subset V(H)\setminus X$, we have
\[e_H(X)=e_{H'}(X)\leq e_{H'}(\overline X)=e_{H}(V(H')\setminus X)\leq e_{H}(V(H)\setminus X)=e_H(\overline X),\]
which concludes the proof.

\end{proof}

We recall the assumptions of Theorem~\ref{thm_big_subgraph_contains_expander}. We are given $G$ that is a $\kappa$-expander, and a subgraph $H$ of $G$ satisfying $e(H)\geq (1-\eps) e(G)$. We are looking for a big enough induced subgraph $H^*$ of $H$ that is a $\frac{\kappa}{3}$-expander.\\

The strategy is the following: if $H$ itself is not a $\frac{\kappa}{3}$-expander, we remove from it a "bad set" of vertices that has a low edge expansion. If we do not obtain a $\frac{\kappa}{3}$-expander, we go on and remove another bad set, and we keep going until we have a  $\frac{\kappa}{3}$-expander. We will show that the process terminates and that the total size of the bad sets is small.

\paragraph{Setup} We will construct a sequence $X_1,X_2,…$ of disjoint subsets of $V(H)$. We first introduce some notation:
$Y_j=\bigcup_{i\leq j} X_i$, $H_j=H[V(H)\setminus Y_{j-1}]$. We will write $\down(X_i)=e_H(X_i,V(H_{i+1}))$, $\dif(X_i)=e_G(X_i)-e_H(X_i)$ and $\out(X_i){ = e_{G\setminus H} (X_i, V(G) \setminus X_i)}${, i.e.\ $\out(X_i)$ will be the number of edges of $G$ that do not belong to $H$ with exactly one endpoint in $X_i$.} We also write
\[\Delta_j=\sum_{i\leq j-1} \dif(X_i)\]
and
\[O_j=\sum_{i\leq j-1} \out(X_i).\]
See Figure~\ref{fig_process} for an illustration of the process.

\begin{figure}
\center
\includegraphics[scale=0.6]{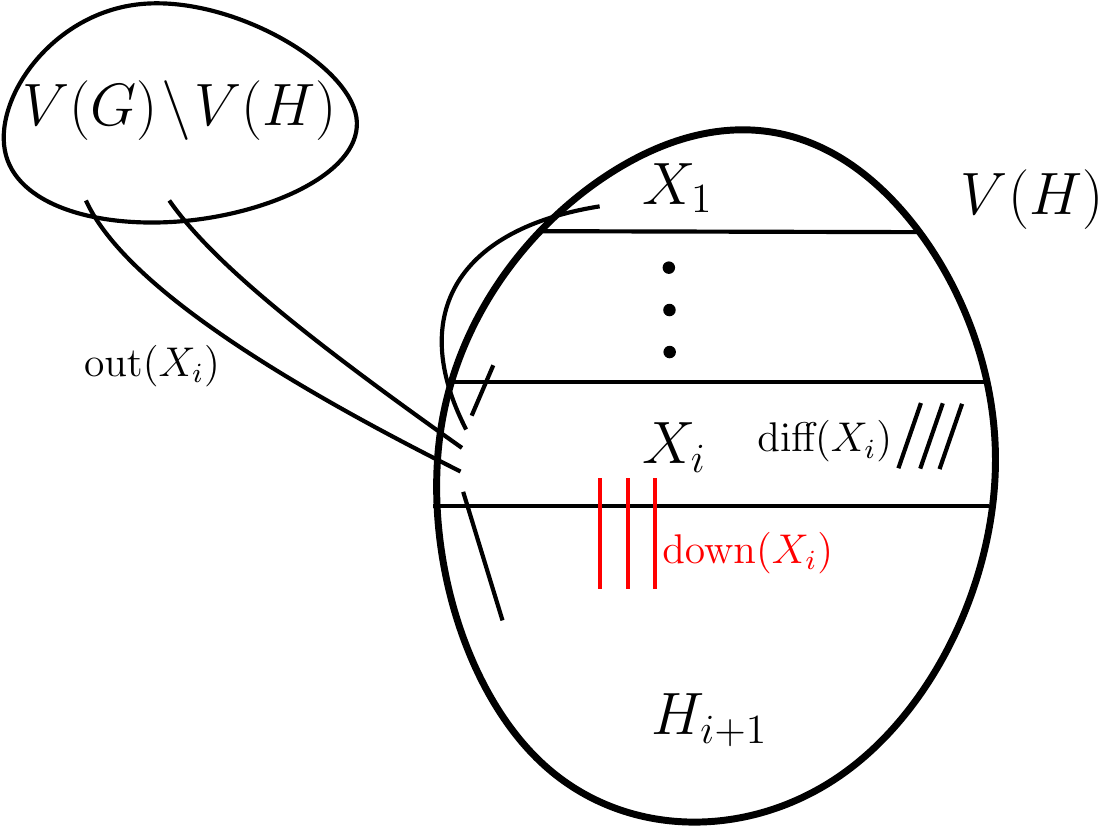}
\caption{Illustration of the process: $H$ is a subgraph of $G$ with red edges depicting those in $H$ and black edges those in $G\setminus H$.}\label{fig_process}
\end{figure}
\paragraph{The process} Now we can define our process. At step $i$, if $H_{i}$ is a $\frac{\kappa}{3}$-expander, then stop the process and output $H_{i}$. Otherwise, there must exist a nonempty subset $X_i$ of $V(H_{i})$ with \begin{equation*}%\label{eq_take_small_one} 
\vol_{H_{i}}(X_i)\leq \vol_{H_{i}}(\overline X_i )\end{equation*} such that
\begin{equation}\label{eq_ineq_down}
\down(X_i)< { \frac{\kappa}{3} \vol_{H_{i}}(X_i) } \leq \frac{\kappa}{3}  \vol_{G}(X_i).
\end{equation}

It is immediate this process terminates because $H_i$ gets strictly smaller at each step.

\paragraph{Lower bounding the size of the output} Fix some $j$.
First, note that $O_j$ counts edges that are in~$G$ but not in $H$ (each edge can be counted twice), and same for $\Delta_j$ (each edge is only counted once this time). The edges counted by $O_j$ and $\Delta_j$ are disjoint, hence
\begin{equation}\label{eq_error_term}
    O_j/2+\Delta_j\leq \eps e(G).
\end{equation}

Also, notice that 
\begin{equation}\label{eq_decomposing_eH}
    e_G(X_i,\overline X_i)=\down(X_i)+\sum_{k<i}e_H(X_i, X_k)+\out(X_i).
\end{equation}

For all $i$, $H_i$ is an induced subgraph of $H$, hence by Lemma~\ref{lem_induced_vol} we have $\vol_H(X_i)\leq \vol_H(\overline X_i)$, and therefore $\vol_G(X_i)-2\dif(X_i)\leq \vol_G(\overline X_i)$.
Since $G$ is a $\kappa$-expander we have
\[\kappa(\vol_G(X_i)-2\dif(X_i))\leq e_G(X_i,\overline X_i)\]
which implies, since $\kappa\leq 1$,
\begin{equation}\label{eq_ineq_vol_diff}\kappa\vol_G(X_i)\leq e_G(X_i,\overline X_i)+2\dif(X_i).\end{equation}

Observe by~\eqref{eq_decomposing_eH} and \eqref{eq_ineq_vol_diff} respectively
\begin{equation*}
\sum_{k<i}e_H(X_k,X_i) +\out(X_i)+2\dif(X_i) = e_G(X_i, \overline X_i) - \down(X_i) + 2\dif(X_i) \geq \kappa \vol_G(X_i) - \down(X_i).
\end{equation*}

Hence by~\eqref{eq_ineq_down} applied twice
\begin{equation}
    \notag \sum_{k<i}e_H(X_k,X_i)+\out(X_i)+2\dif(X_i)\geq \frac{2\kappa}{3}\vol_G(X_i)\geq 2\down(X_i).
\end{equation}

Summing over all $i\leq j-1$, we obtain
\begin{equation}\label{eq_ineq_sum}
    \sum_{k<i\leq j-1}e_H(X_k,X_i)+O_j+2\Delta_j\geq \frac{2\kappa}{3}\vol_G(Y_{j-1})\geq 2\sum_{i\leq j-1} \down(X_i).
\end{equation}

Now, for all fixed $k$, note that 
\begin{equation} 
\notag \sum_{k<i}e_H(X_k,X_i)\leq \down(X_k),\end{equation}
hence the first inequality of~\eqref{eq_ineq_sum} implies
\begin{equation}\label{eq_A}
    \frac{2\kappa}{3} \vol_G(Y_{j-1}) \leq \sum_{k\leq j-2}\down(X_k)+O_j+2\Delta_j,
\end{equation}
and the second inequality of~\eqref{eq_ineq_sum} implies
\begin{equation}\label{eq_B}
O_j +2\Delta_j\geq 2\sum_{k\leq j-1} \down(X_k)-\sum_{k\leq j-2}\down(X_k)\geq \sum_{k\leq j-2} \down(X_k).
\end{equation}
Combining~\eqref{eq_A} and~\eqref{eq_B}, we obtain
\begin{equation}\label{eq_estimation_final_thm2}
    \vol_G(Y_{j-1})\leq \frac{3}{\kappa} (O_j+2\Delta_j)\leq \frac{6}{\kappa}\eps e(G)
\end{equation}
where the second inequality comes from~\eqref{eq_error_term}.\\

Now we can prove Theorem~\ref{thm_big_subgraph_contains_expander}:

\begin{proof}[Proof of Theorem~\ref{thm_big_subgraph_contains_expander}]
Take $G$ and $H$ satisfying the assumptions of the theorem. Now, run the process until it stops at time $j$ and outputs $H_j$. Take $H^*=H_j$. It is a $\frac{\kappa}{3}$-expander by definition of the process. We have $e(H^*)\geq e(G)-\vol_G(Y_{j-1})$, which by~\eqref{eq_estimation_final_thm2} finishes the proof.

\end{proof}

\needspace{6\baselineskip}
\section{From topological minors to subgraphs}\label{sec_topminor_subgraph}
In this section, we prove Proposition~\ref{prop_topminor_to_subgraph}. We start by proving that replacing edges by paths of bounded length in an expander still yields an expander.

\needspace{4\baselineskip}
\begin{lem}\label{lem_smoothing_kappa}
Let $M\geq 2$ and let $H$ be a $\kappa$-expander. Suppose ${G}$ is obtained from $H$ by replacing each edge by a path with at most $M$ edges. Then ${G}$ is a $\frac{\kappa}{2M-1}$-expander.
\end{lem}

\begin{figure}
\center
\includegraphics[scale=0.5]{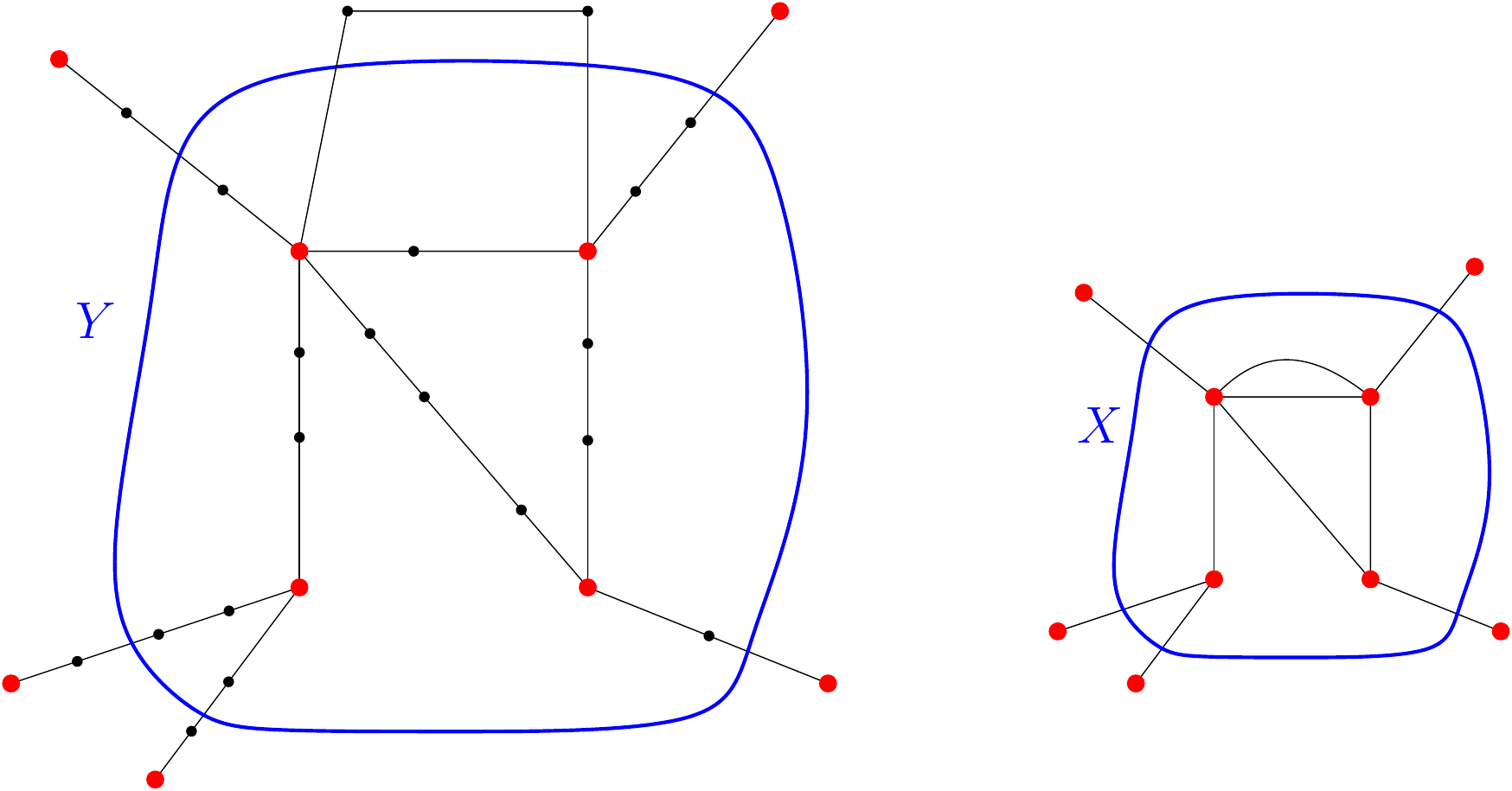}
\caption{Comparing the edge expansions of $Y$ in $G$ and $X$ in $H$ in Lemma~\ref{lem_smoothing_kappa}. Here, $M=3$.}\label{fig_go_big}
\end{figure}

\begin{proof} In ${G}$, colour in red the vertices that come from $H$, and the rest in black. 
Let $Y$ be a subset of $V({G})$ such that  ${G}[Y]$ is connected.  
Let $X$ be the set of red vertices in $Y$. We want to lower bound $h_{G}(Y)$ in terms of $h_H(X)$. See Figure~\ref{fig_go_big} for an illustration. \\

If $X=\emptyset$, then $G[Y]$ is a path on $\leq M-1$ vertices and so $\vol_G(Y)\leq 2(M-1)$ and $e_G(Y,\overline{Y})=2$. Hence 
\[h_{G}(Y)\geq \frac{e_G(Y,\overline{Y})}{\vol_G(Y)} \geq \frac{1}{M-1}.\]

Similarly if $\overline{X}=\emptyset$ then $h_G(Y)=h_G(\overline{Y})\geq 1/(M-1)$. \\

Let us consider now the case $X\neq\emptyset$ and $\overline{X}\neq \emptyset$. The number of edges of $H$ which are incident to a vertex of $X$ is $e_H(X)+e_H(X,\overline{X})\leq \vol_H(X)$. Each edge of $H$ is replaced by a path with at most~$M$ edges, therefore the number of black vertices in $Y$ can be bounded above by~$(M-1)\vol_H(X)$. All these vertices have degree~$2$, by definition of~$G$, hence
\begin{equation}\label{e1}
    \vol_G(Y)\leq\vol_H(X)+2 (M-1) \vol_H(X)=(2M-1)\vol_H(X)
\end{equation}

and similarly
\begin{equation}\label{e2}
    \vol_G(\overline Y)\leq (2M-1)\vol_H(\overline X).
\end{equation}

Now, each edge counted in $e_H(X,\overline{X})$ corresponds to a path in $G$ between $Y$ and
$\overline{Y}$, therefore
\begin{equation}\label{e3}
  e_{G}(Y,\overline{Y}) \geq e_H(X,\overline{X}).  
\end{equation}

It is easy to check that $a' \leq c a$ and $b'\leq cb$ together imply $\min\{a', b'\}\leq c \min\{a,b\}$ and therefore, by \eqref{e1}, \eqref{e2} and~\eqref{e3}:
\[h_G(Y) \geq \frac{1}{2M-1}h_H(X).\]
This concludes the proof.

\end{proof}

\needspace{6\baselineskip}
The next proposition helps us find large subgraphs without long blue paths in vertex-coloured graphs. Recall a vertex coloured graph $G$ has a subset of its degree two vertices coloured blue and $\red(G)$ is obtained from $G$ by smoothing all its blue vertices.

\begin{prop}\label{prop_topminor_find_big_component}
For all $\kappa,\eps,\alpha>0$, there exists $M$ such that the following holds for all vertex-coloured graphs. Let $G$ be such that $\red(G)$ is a $\kappa$-expander and $e(\red(G))\geq \alpha e(G)$, let $G_M$ be the {induced} subgraph of $G$ obtained by deleting all the blue {vertices in} blue paths of length greater than $M$, and let $C_M$ be a connected component of $G_M$ such that $e(\red(C_M))$ is maximal. Then
\[e(\red(C_M))\geq (1-\eps)e(\red(G)),\] and $\red(C_M)$ is a subgraph of $\red(G)$.
\end{prop}

{Note that $\red(C_M)$ need not be an induced subgraph of $\red(G)$. A path of in $G$ of blue vertices of length greater than $M$ from $C_M$ to itself will form an edge in $\red(G)$ but not in $\red(C_M)$.}

\begin{proof}%[Proof of \ref{prop_topminor_find_big_component}
The last point of the proposition follows  from the fact that $C_M$ is a subgraph of $G$. See Figure~\ref{fig_components} for an illustration of this proof.\\

\begin{figure}
\center
\includegraphics[scale=0.5]{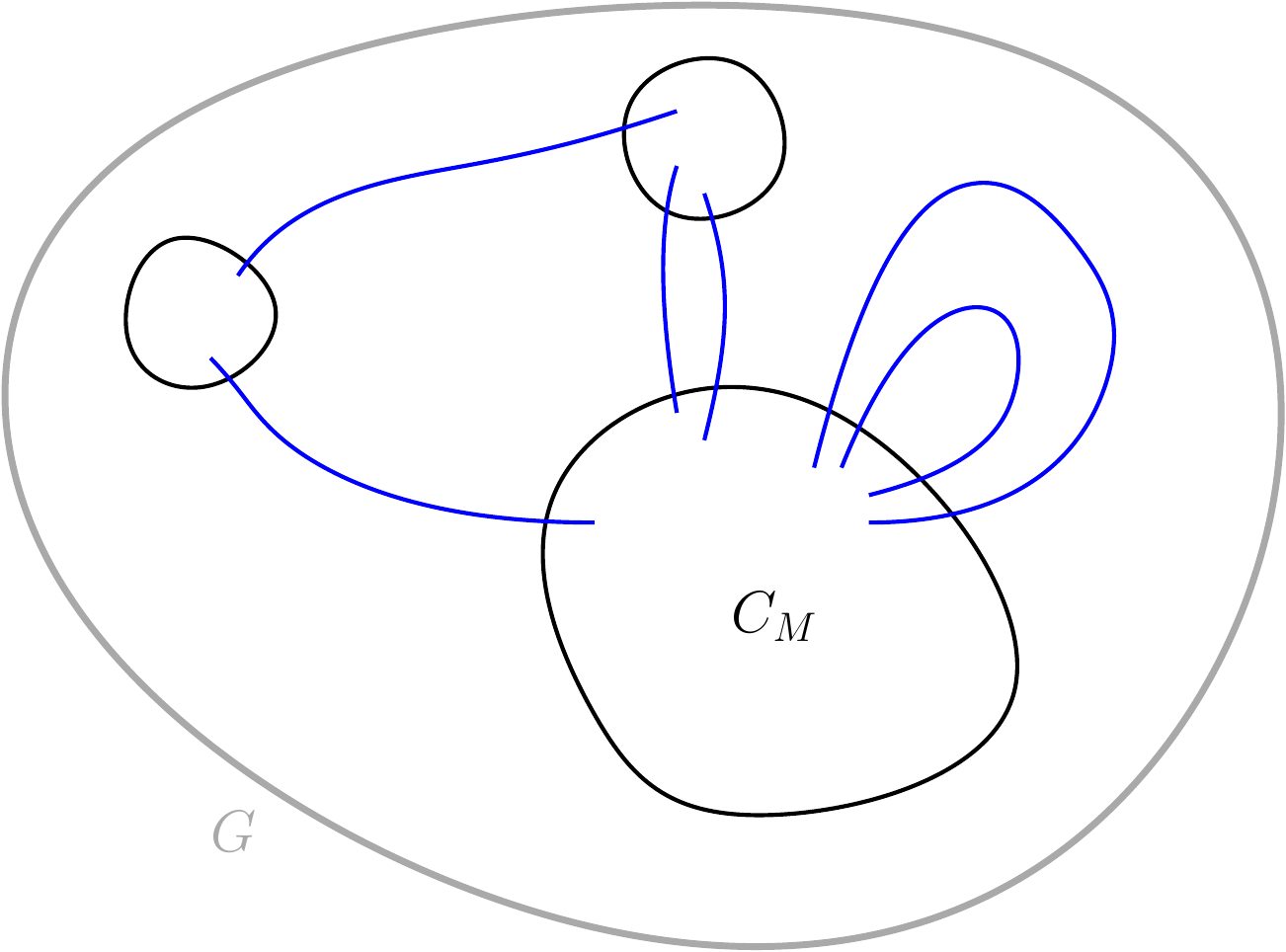}
\caption{In black, the components of $G_M$, in blue, the blue paths of length $>M$.}\label{fig_components}
\end{figure}
Let $M$ be the smallest integer satisfying
\begin{equation}\label{eq_ineq_M}
M>\frac{\alpha}{1-\alpha}\frac{1+\frac{1}{\kappa}}{\eps}.
\end{equation}

Let $p_M$ be the number of blue paths of size $> M$ in $G$. The quantity $e(G)-e(\red(G))$ counts exactly the number of blue vertices in $G$, hence 
\[Mp_M\leq e(G)-e(\red(G)).\]
Now, using the inequality $e(G)\leq \frac 1\alpha e(\red(G))$, we obtain 
\[Mp_M\leq \frac{1-\alpha}{\alpha}e(\red(G))\]

which by~\eqref{eq_ineq_M} implies
\begin{equation}\label{eq_ineq_pM}
p_M\leq \frac{\eps}{1+\frac{1}{\kappa}}e(\red(G)).
\end{equation}
We also have
\begin{equation}\label{eq_GM_pM}
e(\red(G))=e(\red(G_M))+p_M.
\end{equation}
Let $\mathcal{C}(G_M)$ denote the set of connected components in $G_M$. 
Notice that $C$ being a subgraph of $G$ implies that $\red(C)$ is a subgraph of $\red(G)$, thus $\overline{\red(C)}=\red(G)\setminus \red(C)$ is well-defined. Consider  $C\in\mathcal{C}(G_M)$. In $\red(G)$, any edge connecting $\red(C)$ to $\overline{\red(C)}$ corresponds to a blue path of size $>M$ in $G$. Hence 
\begin{equation}\label{eq_out}
\sum_{C\in \mathcal{C}(G_M)} e_{\red(G)}(\red(C), \overline{\red(C)})\leq 2p_M
\end{equation}
 (this is not an equality since some blue paths of size $>M$ might go from a component $C\in\mathcal{C}(G_M)$ to itself, see Figure~\ref{fig_components}).\\

For every connected component $C\neq C_M$ of $G_M$, we have \[e(\red(C))\leq e(\overline{\red(C)}),\] and since $\red(G)$ is a $\kappa$-expander, we have:
\[ e_{\red(G)}(\red(C), \overline{\red(C) )} \geq 2\kappa e(\red(C))\]
because the volume in $\red(G)$ of the vertices of $\red(C)$ is bigger than $2e(\red(C))$.
If we sum over all $C\neq C_M$, we have, by~\eqref{eq_out}
\[p_M\geq \kappa (e(\red(G_M))-e(\red(C_M)))\]
therefore, by~\eqref{eq_GM_pM}
\[e(\red(C_M))\geq e(\red(G))-(1+\frac{1}{\kappa})p_M\]
and because of~\eqref{eq_ineq_pM} we obtain
\[e(\red(C_M))\geq (1-\eps)\: e(\red(G)),\]
which concludes the proof.

\end{proof}

The following lemma helps us invert the $\red$ operation in some sense.

\begin{lem}\label{lem_lifting}
Let $G$ be a vertex-colored graph, and $H$ and induced subgraph of $\red(G)$, then there exists a vertex-colored graph $H^*$ that is an induced subgraph of $G$ such that $\red(H^*)=H$.
\end{lem}
\begin{figure}
    \centering
    \includegraphics{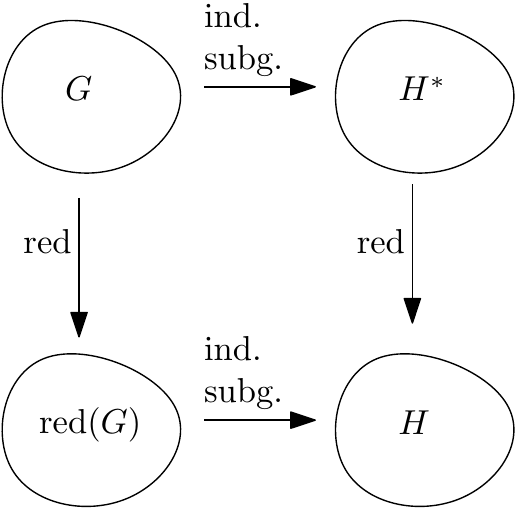}
    \caption{Building $H^*$ out of $H$ in the proof of Lemma~\ref{lem_lifting}. An "ind. subg." arrow means that the target is an induced subgraph of the source. A "red" means that the target is obtained from the source by applying the operation red.}
    \label{fig_lifting}
\end{figure}
\begin{proof}
See Figure~\ref{fig_lifting} for visual aid. Let $X=V(G)\setminus V(H)$. This set can be seen as a subset of $V(G)$. Let us build a set $Y\subset V(G)$ from $X$ in the following way:
\begin{itemize}
    \item a black vertex belongs to $Y$ iff it belongs to $X$,
    \item Notice that every blue vertex is part of a blue path that is incident to two black vertices. A blue vertex belongs to $Y$ iff at least one of these two black vertices belongs to X.
\end{itemize}
Now, let $H^*$ be the graph obtained from $G$ by deleting all vertices of $Y$. We claim that $H^*$ is a vertex-colored graph and that $\red(H^*)=H$. To show that $H^*$ is vertex colored, we need to show that all its blue vertices have degree $2$, or equivalently that any blue vertex of $G$ that is incident to a vertex of $Y$ is also in $Y$ (loosely speaking, if you delete a neighbor of a blue vertex, you also delete this vertex). This last condition is ensured by the construction of $Y$. Finally, $\red(H^*)=H$ is immediate because $V(H)\subset V(H^*)$ and $V(H^*)\setminus V(H)$ only contains blue vertices (again, here we see $V(H)$ as a subset of $V(G)$).

\end{proof}

We are finally ready to prove Proposition~\ref{prop_topminor_to_subgraph}.

\begin{proof}[Proof of Proposition~\ref{prop_topminor_to_subgraph}]
\begin{figure}
    \centering
    \includegraphics{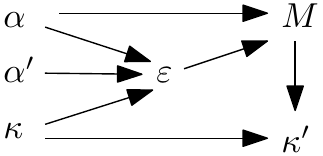}
    \caption{The causal graph of the variables we introduced in the proof of Proposition~\ref{prop_topminor_to_subgraph}.}
    \label{fig_causal}
\end{figure}

This proof might be complicated to follow,  Figure~\ref{fig_diagram} presents the relations between the different graphs involved. Also, we introduce a lot of variables, Figure~\ref{fig_causal} presents a causal graph of all the variables to make sure there is no circularity.
Let $G$ and $H$ be graphs satisfying the assumptions of the proposition.
Let $\eps>0$ be such that 

\begin{equation}\label{eq_def_eps}
\alpha\left(1-\frac{6\eps}{\kappa}\right)=\alpha'.
\end{equation}
By Lemma~\ref{lem_ordering_operations}, there must exist a vertex-coloured graph $G_1$ such that $G_1$ is a subgraph of $G$ (if we forget about the colouring), and such that $\red(G_1)=H$. \\

Now, by Proposition~\ref{prop_topminor_find_big_component}, there exists $M$ depending only on $\kappa$, $\alpha$ and $\eps$, and a subgraph $C_M$ of~$G_1$ such that $e(\red(C_M))\geq(1-\eps)e(H)$ and such that $C_M$ does not contain any blue path of length~$>M$.\\

Hence we have $\red(C_M)$ a subgraph of $H$, with $e(\red(C_M))\geq(1-\eps)e(H)$ and $H$ a $\kappa$-expander. Thus by Theorem~\ref{thm_big_subgraph_contains_expander} and \eqref{eq_def_eps}, there exists an induced subgraph $H'$ of $\red(C_M)$, such that $e(H')\geq \alpha' e(G)$, such that $H'$ is a $\frac \kappa 3$-expander.\\

Now, by Lemma~\ref{lem_lifting}, there exists a graph $H^*$ that is a subgraph of $C_M$ such that $\red(H^*)=H'$.
It is direct to see that $H^*$ does not contain any blue path of length $>M$. Therefore, by Lemma~\ref{lem_smoothing_kappa}, $H^*$ is a $\kappa'$-expander, where 
\[\kappa'=\frac{\kappa}{3(2M-1)}.\]

Furthermore, we have $e(H^*)\geq e(H')\geq \alpha' e(G)$, which concludes the proof.
\begin{figure}
    \centering
    \includegraphics{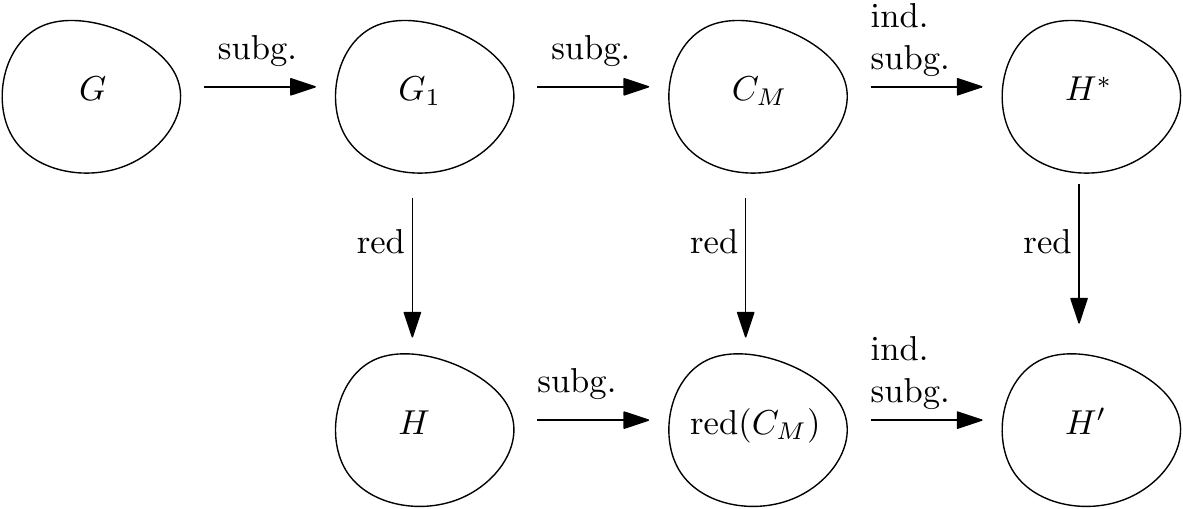}
    \caption{The relations between the different graphs considered in the proof of Proposition~\ref{prop_topminor_to_subgraph}. An "ind. subg." arrow means that the target is an induced subgraph of the source. A "subg." arrow means that the target is a subgraph of the source. A "red" means that the target is obtained from the source by applying the operation red.}
    \label{fig_diagram}
\end{figure}

\end{proof}

\needspace{6\baselineskip}
\section{From subgraph to induced subgraph}\label{sec_subgraph_induced}
Here we prove Proposition~\ref{prop_subgraph_to_induced}. The structure of the proof is similar to the proof of Proposition~\ref{prop_topminor_to_subgraph}. However, it is a bit more technical as the arguments are more intertwined than in the previous section.\\

The following lemma is similar to Lemma~\ref{lem_smoothing_kappa}, but for adding edges instead of paths.

\begin{lem}\label{lem_deleting_edge_kappa}
Let $H$ be a $\kappa$-expander, and let $G$ be obtained from $H$ adding some edges between its vertices, such that no vertex sees its degree multiplied by more than $M$. Then $G$ is a $\frac{\kappa}{M}$-expander.
\end{lem}

\begin{proof}
We stress the fact that $V(G)=V(H)$. 
For all $X\subset V(G)$, we have immediately
\[\vol_G(X)\leq M\vol_H(X)\]
and
\[e_G(X,\overline X)\geq e_H(X,\overline X),\]

therefore the lemma follows.

\end{proof}

Recall an edge-coloured graph $G$ has a subset of its edges coloured blue, $\red(G)$ is obtained from~$G$ by deleting all its blue edges and for $v\in V(G)$ we denote by $d_{\red}(v)$ the degree of $v$ in $\red(G)$.

\begin{lem}\label{lem_few_bad_vertices}
For all $\eps,\alpha>0$, there exists $M$ such that the following holds for all edge-coloured graphs. If $G$ is  such that $e(\red(G))\geq \alpha e(G)$, let $V^*=\{v\in V(G) \mid \deg(v)\geq M d_\red(v)\}$. Then
\[\vol_{\red(G)}(V^*)\leq \eps e(\red(G)).\]
\end{lem}
\begin{proof}
This is immediate because
\[\vol_{\red(G)}(V^*)\leq \frac{1}{M}\vol_G(V^*)\leq \frac{1}{M}e(G)\leq \frac{1}{M\alpha}e(\red(G)).\]

\end{proof}

The following proposition will play the same role as Theorem~\ref{thm_big_subgraph_contains_expander} and Proposition~\ref{prop_topminor_find_big_component} in Section~\ref{sec_topminor_subgraph}. Unfortunately, this time we cannot make it into two separate arguments, which makes things a little more technical.

\begin{prop}\label{prop_algo_induced}
For all $\kappa,\eps,\alpha>0$, there exists $M$ such that the following holds for all edge-coloured graphs. If $G$ is  such that $\red(G)$ is a $\kappa$-expander and $e(\red(G))\geq \alpha e(G)$, then there exists an induced subgraph $G^*$ of $G$ such that $\red(G^*)$ is a $\frac \kappa 3$-expander, and for every vertex $v$ of~$G^*$, $\deg(v)\leq 3M d_\red(v)$, and $e(G^*)\geq \left(1-\eps\left (1+\frac{3}{\kappa}\right)\right)e(\red(G)).$
\end{prop}

\begin{proof}
We will construct an algorithm which takes as input a graph $G$ satisfying the assumptions of the proposition, and outputs a graph $G^*$ with the required properties.

\paragraph{Setup}
Some notation before we start: we will construct a finite sequence $X_0, X_1,…$ of disjoint subsets of $V(G)$. We will write $Y_j=\bigcup_{0\leq i\leq j}X_i$, $V_j=V(G)\setminus Y_{j-1}$ and $G_j=G[V_j]$. We will also write $H=\red(G)$ and $H_j=\red(G_j)$.\\

Finally, let $\upp(X_i)=e_H(X_i,Y_{i-1})$ and $\down(X_i)=e_H(X_i,V_{i+1})=e_{H_i}(X_i,V_{i+1})$ (the second equality is a consequence of the definitions). See Figure~\ref{fig_process_induced} for an illustration of the process.

\begin{figure}
\center
\includegraphics[scale=0.6]{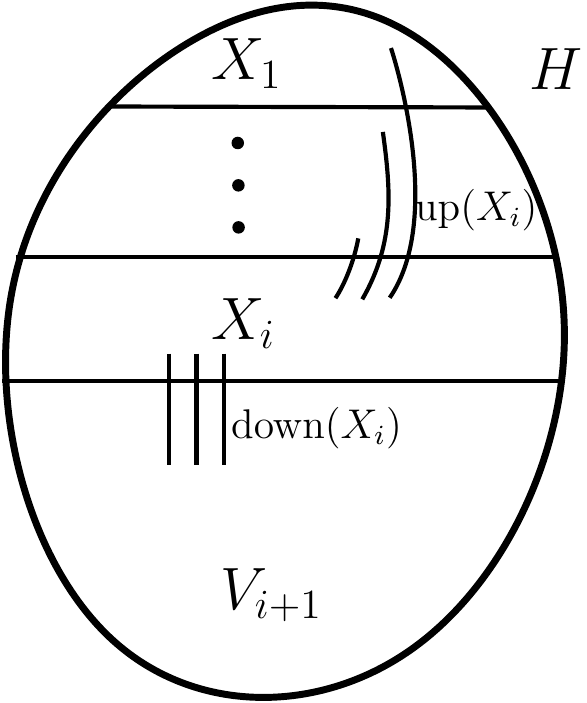}
\caption{Illustration of the process in the proof of Proposition~\ref{prop_algo_induced}.}\label{fig_process_induced}
\end{figure}

\paragraph{The algorithm}
The algorithm works this way:

At step $0$, let $V^*$ and $M$ be given by Lemma~\ref{lem_few_bad_vertices}. Set $X_0:=V^*$.

At step $i>0$:

\textbf{Case 1:} \textit{If} there exists a set $X\subset V_{i}$ such that 
\[\vol_{G_i}(X)\geq 3M\vol_{H_i}(X), \]
set $X_i:=X$.

\textbf{Case 2:} \textit{Else, if} there exists a set $X\subset V_{i}$ such that $\vol_{H_{i}}(X)\leq \vol_{H_{i}}(\overline X)$ and
\[e_{H_i}(X,\overline{X})\leq \frac{\kappa}{3}\vol_{H_{i}}(X),\]
set $X_i:=X$.

\textit{Else}, return $G_{i}$.

Note that this algorithm terminates because $G_{i}$ gets strictly smaller at each step.

\paragraph{Estimating the size of the output graph:} We let the algorithm run until it stops and returns some $G_j$. Let us analyse each step of the algorithm. Take a step $i>0$.

\textbf{If we are in case 1}, we have
\[\vol_{G_i}(X_i)\geq 3M\vol_{H_i}(X_i)\] but
\[\vol_G(X_i)\leq M \vol_H(X_i)\] by definition of $X_0$.

Hence, knowing that $\vol_{H}(X_i)=\upp(X_i)+\vol_{H_i}(X_i)$, and $\vol_{G_i}(X_i)\leq \vol_G(X_i)$ we obtain the following inequality
\begin{equation*}%\label{eq_case1_up_vol}
\upp(X_i)\geq \frac{2}{3}\vol_H(X_i).
\end{equation*}
Using the fact $\vol_{H_i}(X_i)\geq \down(X_i)$, we can obtain the following inequality
\begin{equation*}%\label{eq_case1_up_down}
    \upp(X_i)\geq 2\down(X_i).
\end{equation*}

\textbf{If we are in case 2}, we have
\[\down(X_i)\leq \frac{\kappa}{3}\vol_{H_i}(X_i)\] but, $H$ is a $\kappa$-expander, and by Lemma~\ref{lem_induced_vol} $\vol_{H_i}(X_i)\leq \vol_{H_i}(\overline X_i)$ implies $\vol_{H}(X_i)\leq \vol_{H}(\overline X_i)$, hence
\[\upp(X_i)+\down (X_i)\geq \kappa \vol_H(X_i).\]
Since $\vol_{H_i}(X_i)\leq \vol_H(X_i)$, this implies immediately 
\begin{equation*}%\label{eq_case2_up_vol}
\upp(X_i)\geq \frac{2\kappa}{3}\vol_H(X_i)
\end{equation*}
and
\begin{equation*}%\label{eq_case2_up_down}
    \upp(X_i)\geq 2\down(X_i).
\end{equation*}

\textbf{In any case}, we have for all $i>0$
\begin{equation}\label{eq_up_down}
    \upp(X_i)\geq 2\down(X_i),
\end{equation}
and since $\kappa\leq 1$,
\begin{equation}\label{eq_up_vol}
\upp(X_i)\geq \frac{2\kappa}{3}\vol_H(X_i).
\end{equation}

Finally, let \[S_j=\sum_{0<i\leq j-1} \upp(X_i).\]
By definition of $\upp$ and $\down$, we have
\[S_j\leq \sum_{0\leq i \leq j-2} \down(X_i),\]
which, by \eqref{eq_up_down} implies 
\[S_j\leq \down(X_0)+\frac{1}{2}S_{j-1}.\]
Using the trivial inequalities $S_{j-1}\leq S_j$ and $\down(X_0)\leq\vol_H(X_0)$, we obtain
\begin{equation}\label{eq_Sj_X0}
   S_j\leq 2  \vol_H(X_0).
\end{equation}
On the other hand, by summing \eqref{eq_up_vol} over $0<i\leq j-1$, we have for $j\geq 2$
\begin{equation}\label{eq_vol_Y_X}
    S_j\geq \frac{2\kappa}{3}\Big(\vol_H(Y_{j-1})-\vol_H(X_0)\Big).
\end{equation}
Combining~\eqref{eq_Sj_X0} and~\eqref{eq_vol_Y_X}, one obtains for $j\geq 2$
\begin{equation}\label{eq_vol_final}
    \vol_H(Y_{j-1})\leq \left(1+\frac{3}{\kappa}\right)\vol_H(X_0).
\end{equation}

Notice that if $j=1$ then $Y_{j-1}=X_0$ and so \eqref{eq_vol_final} holds for all $j$. By Lemma~\ref{lem_few_bad_vertices}, we have $\vol_H(X_0)\leq \eps e(H)$. 
Now, since $e(G_j)\geq e(H_j)\geq e(H)-\vol_H(Y_{j-1})$, by~\eqref{eq_vol_final} we have
\[e(G_j)\geq \left(1-\eps\left (1+\frac{3}{\kappa}\right)\right)e(H).\]
The condition of case 1 ensures that for all $v$ of $G_j$, $\deg(v)\leq 3M d_\red(v)$, and the condition of case 2 ensures that $\red(G_j)$ is a $\frac \kappa 3$-expander. We set $G^*=G_j$, this finishes the proof.

\end{proof}

We can now turn to the proof of Proposition~\ref{prop_subgraph_to_induced}
\begin{proof}[Proof of Proposition~\ref{prop_subgraph_to_induced}]
Let $G$ and $H$ be graphs satisfying the assumptions of the proposition. Let $\eps>0$ be such that 

\begin{equation*}%\label{eq_def_eps_induced}
\alpha\left(1-\eps\left(1+\frac{3}{\kappa} \right)\right)=\alpha'.
\end{equation*}
By Lemma~\ref{lem_ordering_operations}, there must exist an edge-coloured graph $G_1$ such that $G_1$ is an induced subgraph of $G$ (if we forget about the colouring), and such that $\red(G_1)=H$.\\

Now, by Proposition~\ref{prop_algo_induced}, there exists $M$ depending only on $\kappa$, $\alpha$ and $\eps$, and an induced subgraph~$H^*$ of $G_1$ such that $e(H^*)\geq\left(1-\eps\left(1+\frac{3}{\kappa} \right)\right)e(H)\geq \alpha' e(G)$ and such that for every $v \in V(H^*)$, we have 
\begin{equation}\label{eq_cond_H*}
  \deg(v)\leq 3M d_\red(v),  
\end{equation}
and such that $\red(H^*)$ is a $\frac \kappa 3$-expander.\\

By~\eqref{eq_cond_H*} and Lemma~\ref{lem_deleting_edge_kappa}, $H^*$ is a $\kappa'$-expander, with $\kappa'=\frac{\kappa}{9M}$, which concludes the proof.

\end{proof}

\needspace{4\baselineskip}
\bibliographystyle{abbrv}
\bibliography{bibli}

\end{document}